\documentclass[12pt,reqno]{amsart}

\usepackage{epsfig}
\usepackage{varioref}

\usepackage{amsmath}

\usepackage{amsfonts}

\usepackage{amsthm}

\usepackage{amssymb}

\def\sqr#1#2{{\vcenter{\vbox{\hrule height .#2pt
     \hbox{\vrule width .#2pt height#1pt \kern#1pt \vrule
     width .#2pt} \hrule height .#2pt}}}}

\newcommand{\R}{{\mathbb{R}}}

\newcommand{\E}{{\mathbb{E}}}

\newcommand{\C}{{\mathbb{C}}}

\newcommand{\N}{{\mathbb{N}}}

\newcommand{\Po}{{\mathcal{P}}}

\newcommand{\Ba}{{\mathcal{B}}}

\newcommand{\Hi}{{\mathcal{H}}}

\newcommand{\Fo}{{\mathcal{F}}}

\newcommand{\Rea}{\operatorname{Re}}

\newcommand{\Mi}{{\mathcal{M}}}

\newcommand{\di}{{\mathcal{D}}}

\newtheorem{teorema}{Theorem}

\newtheorem{Remark}{Remark}

\newtheorem{definition}{Definition}

\newtheorem{proposition}{Proposition}
\newtheorem{corollary}{Corollary}

\newtheorem{lemma}{Lemma}

\begin{document}

\title[Infinite dimensional oscillatory integrals]{Infinite dimensional oscillatory integrals with polynomial phase and applications to high order heat-type equations}

\author{S. Mazzucchi}
\address{  Dipartimento di Matematica,  Universit\`a di Trento, 38123 Povo, Italia}

\maketitle

\begin{abstract}
The definition of infinite dimensional Fresnel integrals is generalized to the case of polynomial phase functions of any degree and applied to the  construction of a functional integral representation of the solution to a general class of high order heat-type equations.\\

\bigskip


\end{abstract}

\vskip 1\baselineskip

\section{Introduction}
Functional integration is a powerful tool for the study of dynamical systems \cite{Sim}. The main example is the celebrated Feynman-Kac formula  \eqref{Fey-Kac}, which provides  a probabilistic representation of the solution to the heat equation
\begin{equation} \label{heat}
\left\{\begin{aligned}
\frac{\partial}{\partial t}u(t,x)&=\frac{1}{2}\Delta u(t,x) -V(x)u(t,x),\qquad t\in \R^+, x\in\R^d, \\
u (0,x)&=u _0(x), 
\end{aligned}\right. \end{equation}
in terms of the expectation with respect to the distribution of the Wiener process $W$ starting at $x$ (see, e.g. \cite{KarSh}),
\begin{equation}\label{Fey-Kac}
u (t,x)=\E^x[ e^{-\int_0^tV(W(s))ds}u_0(W(t))].\end{equation}
Formula \eqref{Fey-Kac} can be established under rather mild requirements on the potential $V$ and the initial datum $u_0$ (see, e.g. , \cite{Sim}) and provides an important instrument in the study of  heat equation and  its solutions. 

More generally, an extensively developed theory relates stochastic processes with the solution to parabolic equations associated to second-order elliptic operators \cite{Dyn}.
However, that theory cannot be applied to  more general  PDEs such as, for instance,  the Schr\"odinger equation
\begin{equation} \label{schroedinger}\left\{ \begin{aligned}
i\frac{\partial}{\partial t}u(t,x)& =-\frac{1}{2}\Delta u(t,x) +V(x)u(t,x),\qquad t\in \R^+, x\in\R^d\\
u (0,x)&=u _0(x)
\end{aligned}\right. \end{equation}
describing the time evolution of the state of a nonrelativistic quantum particle, 
or also heat-type equations associated to high-order differential operators, such as for instance
\begin{equation}\label{eqLapl2}
\frac{\partial}{\partial t} u(t)=- \Delta^2u(t)-V(x)u(t,x).
\end{equation}
Indeed,   a  Markov process $\{X(s)\, :\, 0\leq s\leq t\}$ playing the same role for Eq. \eqref{schroedinger} or Eq. \eqref{eqLapl2}  as the Brownian motion for the heat equation doesn't exist. Hence  there is no ``generalized Feynman Kac formula" 
\begin{eqnarray}\label{Fey-KacN}
u (t,x)&=&\E^x[ e^{-\int_0^tV(X(s))ds}u_0( X(t))]\\
 &=&\int _{\R^[0,t]}e^{-\int_0^tV(\omega(s))ds}u_0( \omega (t))dP(\omega),\nonumber
\end{eqnarray} representing  the solution of  Eq. \eqref{schroedinger} or Eq. \eqref{eqLapl2} in terms of a (Lebesgue type) integral with respect to a probability measure $P$ on $\R^{[0,t]}$  associated to the  process  $X(s)$.\\
Contrarily to the heat equation case, for both   Eq. \eqref{schroedinger} and Eq. \eqref{eqLapl2} the fundamental solution $G_t(x,y)$ is not real and positive, even in the simplest case $V\equiv 0$.  In particular    the Green function $G_t(x,y)$ of the Schr\"odinger equation is complex, while for the  high-order heat-type equation \eqref{eqLapl2}  $G_t(x,y)$ is real and attains both positive and negative values \cite{Hochberg1978}. Therefore it  cannot be interpreted as the density of a transition probability measure.
As a  troublesome consequence,
the complex (resp. signed) finitely-additive   measure $\mu$ on  $\Omega=\R^{[0,t]}$ defined on the algebra of  ``cylinder sets"
$I_k\subset\Omega$ (where $\Omega\equiv \R^{[0,+\infty)}$) of the form  $$I_k:=\{\omega\in \Omega: \omega (t_j)\in [a_j,b_j], j=1,\dots k\},\quad  0<t_1<t_2<\dots t_k,$$ by
\begin{equation}\label{CylMeas}
\mu(I_k)=\int_{a_1}^{b_1}...\int_{a_k}^{b_k}\prod_{j=0}^{k-1}G_{t_{j+1}-t_j}(x_{j+1},x_j)dx_1...dx_{k},
\end{equation} doesn't extend to a corresponding $\sigma$-additive measure on the generated $\sigma$-algebra. As a matter of fact, if this measure existed, it would have infinite total variation.\\
 This problem was addressed  in 1960 by Cameron \cite{Cam} for the Schr\"odinger equation and by Krylov \cite{Kry} for Eq. \eqref{eqLapl2}. These results may be viewed as particular cases of a general theorem later established by E. Thomas \cite{Thomas}, extending Kolmogorov existence theorem to limits of  projective systems of signed or complex measures, instead of probability ones.\\ 
In fact, these no-go results forbid a functional integral representation of the solution of Eq. \eqref{schroedinger} or Eq. \eqref{eqLapl2} in terms of a Lebesgue-type integral with respect to a $\sigma$-additive  complex or signed measure with finite total variation.   Consequently, the integral appearing in the generalized Feynman-Kac formula  \eqref{Fey-KacN} has to be thought in a  weaker sense.  One possibility is the definition of the ``integral" in terms of a linear continuous functional on a suitable Banach algebra of ``integrable functions", in the spirit of Riez-Markov theorem, that provides a one-to-one correspondence between complex bounded measures (on suitable topological spaces $X$) and linear continuous functionals on $C_\infty(X)$ (the continuous functions on $X$ vanishing at $\infty$).\\
Referring to  Schr\"odinger equation, this issue has been extensively studied, producing a number of different mathematical definitions of Feynman path integrals (see \cite{Ma} for an account). We mention in particular  for future reference the {\em Parseval approach}, introduced by It\^o \cite{Ito1,Ito2} in the 60s and developed in the 70s by S. Albeverio and R. Hoegh-Krohn \cite{AlHK, AlHKMa}, and by D. Elworthy and A. Truman \cite{ELT}.

Dealing with the parabolic equation \eqref{eqLapl2} associated to the bilaplacian,  various  formulations have been proposed. One of the first was introduced by  Krylov \cite{Kry} and extended by Hochberg \cite{Hochberg1978}. Defining a suitable stochastic pseudo-process whose transition probability function is not positive definite,  the authors realized formula \eqref{Fey-KacN} in terms of the expectation with respect to a  signed measure on $\R^{[0,t]}$ with infinite total variation. That is the reason way  the integral in  \eqref{Fey-KacN} is not defined in Lebesgue sense, but  is meant as the limit of finite dimensional cylindrical approximations \cite{BeHocOr}.
It is worthwhile  mentioning  the work by D. Levin and T. Lyons relying on the ``rough paths" theory. Indeed, in  \cite{LevLyo} the authors  conjecture  that the signed measure (with infinite total variation)  associated to the Krylov-Hochberg  pseudo-process could became finite if defined on a certain quotient space on the path space (two path paths are equivalent if they differ for reparametrization). \\
A different approach was proposed by  Funaki \cite{Funaki1979} and continued by Burdzy \cite{Burdzy1993}. It is based 
on the construction of a complex-valued stochastic process with dependent increments, obtained by composing two independent Brownian motions. In  \cite{Funaki1979}, formula \eqref{Fey-KacN} with $V=0$  is realized as an integral  with respect to a well defined positive probability measure on a complex space  for a suitable class of analytic initial data $u_0$ at least. These results have been further developed in \cite{Funaki1979,HocOr96,OrZha} and are related to Bochner's subordination theory \cite{Boch}.
% These results are also related to Bochner subordination \cite{Boch}. There are also similarities between the Funaki's process  and the {\em iterated Brownian motion} \cite{Burdzy1993}, but the latter is not connected to the  probabilistic representation of the solution of a partial differential equation with regular coefficients. In fact the processes constructed by iterating copies of independent BMs (or other process) are associated to higher order PDE of particular form, where the initial datum plays a particular role and enters also in the differential equation \cite{Allo}. 
Complex-valued processes, related to PDEs of the form \eqref{eqLapl2}, were also proposed by other authors exploiting various techniques \cite{Burdzy1995,ManRy93,BurMan96,Sainty1992}. A new construction for the solution of a general class of high order heat-type equations has been recently proposed, where formula \eqref{Fey-KacN} has been realized as limit of expectations with respect to a sequence of suitable random walks in  the complex plane \cite{BoMa14}.\\
 We also mention a completely different approach proposed by R. L\'eandre \cite{Lea}, which shares some analogies with the mathematical construction of Feynman path integrals with the  white-noise-calculus approach \cite{HKPS}.\\
It is worthwhile remarking that  most  of the results appearing in the literature are restricted to the cases where either $V=0$ or $V$ is linear. The construction of a generalized Feynman-Kac type formula is still lacking for the solution of high-order heat-type equations similar to \eqref{eqLapl2} with a more general  $V$.  \par
This work aims to construct a  Feynman-Kac formula for the  solution of  a general class of high-order heat-type equations of the form
\begin{equation}\label{PDE-N}
\frac{\partial}{\partial t}u(t,x)=(-i)^p \alpha  \frac{\partial^p}{\partial x^p}u(t,x)+V(x)u(t,x), \quad t\in [0, +\infty), \; x\in \R,
\end{equation}
where $p\in\N$,  $p>2$, $\alpha\in\C$ is a complex constant and $V:\R\to\C$ a continuous bounded function Fourier transform of a complex Borel measure on $\R$. \\
Adopting the {\em Fresnel integral } formulation of the mathematical definition of Feynman path integrals \cite{AlHKMa, AlHK}, we introduce  infinite dimensional Fresnel integrals with polynomial phase, generalizing the existing results valid for quadratic phase functions. 
If the phase function is an homogeneous polynomial of order $p$, we show in particular how this new kind of functional integral
is related to the fundamental solution of Eq. \eqref{PDE-N} with $V\equiv 0$.  This relation will be eventually exploited in the proof of a functional integral representation of the solution of Eq.  \eqref{PDE-N}, for a  suitable class of potentials $V$ and initial data $u_0$, giving rise to a new type of generalized Feynman-Kac formula.

In section 2, a detailed study of the fundamental solution of Eq. \eqref{PDE-N} takes place in the case  $V=0$. In section 3, we introduce the definition of infinite dimensional Fresnel integral with polynomial phase function  showing that a particular example is related to the PDE \eqref{PDE-N} with $V\equiv 0$. In section 4, we build up a representation of the solution of \eqref{PDE-N} with $V\neq 0$ in terms of an infinite dimensional Fresnel integral.
 
\section{The fundamental solution of high-order heat-type equations}
Let us consider the $p$-order heat-type equation:
\begin{equation} \label{PDE-p}\left\{ \begin{aligned}
\frac{\partial}{\partial t}u(t,x)&=(-i)^p \alpha \frac{\partial^p}{\partial x^p}u(t,x)\\
u(0,x) &= u_0(x),\qquad x\in\R, t\in [0,+\infty)
\end{aligned}\right. \end{equation}
where $p\in\N$, $p\geq 2$, and $\alpha\in\C$ is a complex constant. In the following we shall assume that  $|e^{\alpha tx^p}|\leq 1$ for all $x\in\R$ and $ t\in [0,+\infty)$. In particular, if $p$ is even  this condition is fulfilled if  $\Rea(\alpha)\leq 0$, while if $p$ is odd then    $\alpha$  will be taken purely imaginary.\\
In the case where $p=2$ and $\alpha \in \R$, $\alpha<0$, we obtain the heat equation, while for $p=2$ and $\alpha =i$ Eq. \eqref{PDE-p}  is the Schr\"oedinger equation. Since both cases  are extensively studied,   in the following we shall mainly focus ourselves on the case where $p\geq 3$.
 
Let $G^p_t(x,y)$ be the fundamental solution of Eq.\eqref{PDE-p}. Given an initial datum $u_0$ belonging to the space $S(\R)$ of  Schwartz   test functions,  the solution of the Cauchy problem \eqref{PDE-p} is given by:
\begin{equation}\label{Green}
u(t,x)=\int_\R G^p_t(x,y)u_0(y)dy.
\end{equation}
In particular the following equality holds: $$G^p_t(x-y)=g^p_t(x-y),$$ where $g^p_t\in S'(\R)$ is  the Schwartz distribution defined by the Fourier transform
\begin{equation}\label{Green2}
g^p_t(x):=\frac{1}{2\pi}\int e^{ikx}e^{\alpha tk^p}dk,\qquad x\in\R.
\end{equation}

The following lemmas state some regularity properties of the distribution $g^p_t$ that will be used in the next section. 
\begin{lemma}\label{lemma1}
The  tempered distribution \eqref{Green2} is a $C^\infty$ function.
\end{lemma}
\begin{proof}
%$The proof is divided into two steps. We shall first prove the regularity,   then we shall study the behavior of $g_t^p(x)$ as $x\to \infty$.\\
%\noindent{\em 1.Regularity of $g^p_t$: } 
A priori $g^p_t$ is an element of  $S'(\R)$, the Schwartz space of distribution, but we shall prove that $g^p_t$ is a $C^\infty$ function defined by an absolutely convergent Lebesgue integral.  This can be easily proved in the case where $p$ is even and $\Rea(\alpha) <0$, since the function $k\mapsto e^{\alpha tk^p}$ is an element of  $L^1(\R)$.\\
%In  the case where $p$ is even and $\Rea(\alpha) <0$ the function $k\mapsto e^{\alpha tk^p}$ is an element of  $L^1(\R)$ and the  integral \eqref{Green2} is absolutely convergent and defines a $C^\infty$ function of the variable $x\in \R$.
In the case where $\Rea(\alpha)= 0$, i.e. $\alpha=ic$ with $c\in\R$, the function $k\mapsto e^{\alpha tk^p}$ is not summable.
 Let us  denote by $\psi\in S'(\R)$ the tempered distribution defined by this map and by $\chi_{[-R,R]}$  the characteristic function of the interval $[-R,R]\subset \R$.  By the convergence of $\chi_{[-R,R]}\psi$ to $ \psi $ in  $S'(\R)$ as $R\to +\infty$ and the continuity of the Fourier transform as a map from $S'(\R)$ to $S'(\R)$ we have that $$g^p_t=\hat \psi=\lim _{R\to+\infty }\widehat{\chi_{[-R,R]}\psi}.$$
On the other hand, by a change in the integration path in the complex $k$-plane, in the case where $p$ is even and $c>0$ we have:
\begin{eqnarray}
g^p_t(x)&=&\lim_{R\to\infty}\frac{1}{2\pi}\int_{-R}^R e^{ikx}e^{ict k^p}dk=\lim_{R\to\infty}\frac{1}{2\pi}\int_{0}^R (e^{ikx}+e^{-ikx})e^{ict k^p}dk\nonumber\\
&=&\lim_{R\to\infty}\frac{e^{i\pi/2p}}{2\pi}\int_{0}^R (e^{ie^{i\pi/2p}kx}+e^{-ie^{i\pi/2p}kx})e^{-c tk^p}dk\nonumber\\
&=&\frac{e^{i\pi/2p}}{2\pi} \int_\R e^{ie^{i\pi/2p}kx}e^{-ct k^p}dk,\label{G1}
\end{eqnarray}
while in the case where $p$ is even and $c<0$:
\begin{equation}
g^p_t(x)=\frac{e^{-i\pi/2p}}{2\pi} \int_\R e^{ie^{-i\pi/2p}kx}e^{-c tk^p}dk.\label{G2}
\end{equation}
In the case  $p$ is odd, a different integration contour in the complex $k-$plane yields the following representation:
\begin{eqnarray}
g^p_t(x)&=&\lim_{R\to\infty}\frac{1}{2\pi}\int_{-R}^R e^{ikx}e^{ict k^p}dk\nonumber\\
&=&\frac{1}{2\pi}\int_{\R+i\eta} e^{ixz}e^{ic tz^p}dz\label{G3}
\end{eqnarray}
where $\eta>0$ if $c>0$ while $\eta<0$ if $c<0$. The integrand in the second line of \eqref{G3} is absolutely convergent since $ |e^{ict (\Rea(z)+i\eta)^p}|\sim e^{-ct\eta (\Rea(z))^{p-1}}$ as $|\Rea(z)|\to\infty$. \\
Eventually representations \eqref{G1},  \eqref{G2} and \eqref{G3} show that $g^p_t$ is a $C^\infty $ function of the variable $x$.
\end{proof}
\begin{Remark}
The proof of lemma \eqref{lemma1} shows that $g^p_t:\R\to \C$ can be extended to an entire analytic function  of $z\in\C$.
The analyticity of $g^p_t$ follows by the application of Fubini's and Morera's theorems.
\end{Remark}
\begin{Remark}
A formula similar to \eqref{G1} has also been proved in \cite{AlMa2005} and applied to the study of some asymptotic properties of finite dimensional Fresnel integral with polynomial phase function.
\end{Remark}
The following lemma relies on  the study of the detailed asymptotic behaviour of $g^p_t(x)$ for $x\to \infty$.
\begin{lemma}\label{lemma1-asy}
The function $g^p_t$ is bounded. In particular if $p$ is even and $\Rea(\alpha)<0$ then $g^p_t\in L^1(\R)$.
\end{lemma}
\begin{proof}
By lemma \ref{lemma1} the function $g^p_t$ is continuous, hence the proof of  its boundedness can be based only on the study of its asymptotic  behavior for $x\to \infty$. This task is accomplished by means of the stationary phase method \cite{Mur,Hor}.\\ 
%\noindent {\em 2. Behaviour of $g^p_t$ for $x\to\infty$} In the second part of the proof we show that the function $g^p_t:\R\to\C$ belongs to $L^1(\R)$ by studying its asymptotic behavior for $|x|\to \infty $ by means of the stationary phase method \\
For $x\to +\infty$,  a change of variables in \eqref{Green2} gives:
\begin{equation}\label{int1}
g^p_t(x)=\frac{x^{\frac{1}{p-1}}}{2\pi}\int_\R e^{x^{p/p-1}(i\xi+\alpha t\xi^p)}d\xi=\frac{x^{\frac{1}{p-1}}}{2\pi}\int_\R e^{x^{p/p-1}\phi(\xi)}d\xi,
\end{equation}
 $\phi:\R\to \C$ being the complex phase function 
$$\phi(\xi)=i\xi+\alpha t\xi^p, \qquad \xi\in\R.$$
If either $\Rea (\alpha) \neq 0$ or 
$p$ is odd and $\alpha =ic$, with $c\in \R^+$, then the phase function $\phi$ has no stationary points on the real line, i.e. there are no real solutions of the equation $\phi'(\xi)=0$.  In this cases an integration by parts argument yelds: 
$$\int e^{x^{\frac{p}{p-1}}\phi(\xi)}d\xi=\int \frac{1}{x^{\frac{p}{p-1}}\phi'(\xi)} \frac{d}{d\xi} e^{x^{\frac{p}{p-1}}\phi(\xi)}d\xi=\frac{1}{x^{\frac{p}{p-1}}}\int e^{x^{\frac{p}{p-1}}\phi(\xi)}\frac{\phi''(\xi)}{(\phi'(\xi))^2}d\xi.$$
By iterating  this procedure we obtain that for all $N\in \N$: 
$$g^p_t(x)\stackrel{x\to +\infty}{<<}(x^{\frac{p}{p-1}})^{-N}.$$
 
% In the case where $p$ is even and $\Rea(\alpha)<0$, then the result follows easily by the following inequalities:
%$$|G_t(x)|=\frac{1}{2\pi}\left|\int e^{ikx}e^{\alpha tk^p}dk\right|\leq \frac{1}{2\pi}\int e^{\alpha tk^p}dk, \quad %\forall x\in\R.$$

In the case where  $\alpha =ic$ with $c\in \R$, Eq. \eqref{int1} can be written as
\begin{equation}\nonumber
g^p_t(x)=\frac{x^{\frac{1}{p-1}}}{2\pi}\int_\R e^{ix^{p/p-1}(\xi+c t\xi^p)}d\xi, \qquad x>0.
\end{equation}
If $p$ is even, an application of the stationary phase method \cite{Mur,Hor} gives: 
\begin{eqnarray}
g^p_t(x)&=&\frac{x^{\frac{1}{p-1}}}{2\pi}\int_\R e^{ix^{p/p-1}(\xi+ct\xi^p)}d\xi \nonumber\\
&\stackrel{x\to +\infty}{\sim}& e^{sign (c)i\frac{\pi}{4}} \frac{x^{\frac{2-p}{2(p-1)}}}{\sqrt{2\pi}} e^{-ix^{p/p-1}\frac{p-1}{p}\left(\frac{1}{pct}\right)^{1/p-1}} \sqrt{   \frac{ (pct)^{ \frac{p-2}{p-1}}  }{|c|tp(p-1)  }    } .\nonumber
\end{eqnarray}
In the case where $p$ is odd  and $c<0$, the same technique yields:
\begin{eqnarray}
g^p_t(x)&=&\frac{x^{\frac{1}{p-1}}}{2\pi}\int_\R e^{ix^{p/p-1}(\xi+ct\xi^p)}d\xi \nonumber\\
&\stackrel{x\to +\infty}{\sim} & e^{-i\frac{\pi}{4}}(p-1)^{-1/2}(p|c|t)^{-\frac{1}{2(p-1)}} \frac{x^{\frac{2-p}{2(p-1)}}}{\sqrt{2\pi}} e^{ix^{p/p-1}\frac{p-1}{p}\left(-\frac{1}{pct}\right)^{\frac{1}{p-1}}}.\nonumber
\end{eqnarray}

The case where $x\to -\infty$ can be  studied in the same way. In particular,   if $p$ is an even integer the behaviour of $g^p_t$ for $x\to -\infty$  coincides with the one for $x\to +\infty$.\\
 For  $p$  odd  and  
 $x<0$, a change of variable argument gives:
$$g^p_t(x)=\frac{(-x)^{\frac{1}{p-1}}}{2\pi}\int_\R e^{i(-x)^{p/p-1}(-\xi+ct\xi^p)}d\xi. $$
If $c<0$ then the phase function $\phi(\xi)=-\xi+ct\xi^p$ has no real  stationary points, hence
 $$g^p_t(x)\stackrel{x\to -\infty}{<<}x^{-N}, \qquad  \forall N\in\N.$$ In the case where $c>0$ and $x\to -\infty$ the stationary phase method yields
$$g^p_t(x)\stackrel{x\to -\infty}{\sim} e^{i\frac{\pi}{4}}(p-1)^{-1/2}(pct)^{-\frac{1}{2(p-1)}} \frac{(-x)^{\frac{2-p}{2(p-1)}}}{\sqrt{2\pi}}e^{i(-x)^{p/p-1}\frac{1-p}{p}(pct)^{-\frac{1}{p-1}}}.$$

Eventually  these results give the boundedness of the function $g^p_t$. Furthermore, if $p$ is even and $\Rea(\alpha)<0$ then  $g^p_t$ is even  summable.

\end{proof}

%Similar techniques can be applied in the study of small time asymptotics of the function $g^p_t:\R\to \C$, as the following lemma shows.
%\begin{lemma}
%For $t\downarrow 0$ 
%$$g^p_t(x)\sim ????\qquad \forall x\in \R$$
%Moreover

%$$|g^p_t(x)-???|<< $$
%\end{lemma}
%\begin{proof}

%By a change of variables, for $t>0$ one has:
%$$g^p_t(x)=\frac{t^{-\frac{1}{p-1}}}{2\pi}\int_\R e^{    t^{-\frac{1}{p-1}}  \left( i\xi x+\alpha \xi ^p\right)  }d\xi, \qquad x\in \R.$$
%By applying the saddle point method \cite{Mur} (see the appendix for further details) we obtain:

%\end{proof}

\section{Infinite dimensional Fresnel integrals with polynomial phase}\label{sez2}
Classical oscillatory integrals on $\R^n$ are objects of this form
\begin{equation}\label{int-osc}\int _{\R^n}f(x)e^{i\Phi(x)}dx,
\end{equation}
where $\Phi$ and $ f$ are complex Borel functions. The interesting case where the {\em phase function} $\Phi$ is real valued has been extensively studied in connection with the theory of Fourier integral operator \cite{Hor}. If the function $f$ is not summable the integral \eqref{int-osc} is not defined in Lebesgue sense. In \cite{Hor},  H\"ormander proposes and exploits an alternative definition which can handle the case where $f\notin L^1(\R^n)$. We present here a formulation of H\"ormander's definition of oscillatory integral, which was  applied  to the mathematical construction of Feynman path integrals in \cite{ELT,AlBr}.   
\begin{definition}\label{def-int-osc}Let $f:\R^n\to \C$ and $\Phi:\R^n\to \R$ be Borel functions. Assuming that: 
\begin{enumerate}
\item for any Schwartz test function $\phi\in S(\R^n)$ such that $\phi(0)=1$ the function $g_\epsilon (x):=\phi(\epsilon x)f(x)e^{i\Phi(x)}$ is summable,
\item  the limit $\lim_{\epsilon \to 0}\int g_\epsilon (x)dx $ exists and is independent of $\phi$.
\end{enumerate}
Then the {\em oscillatory integral } $ \int^o _{\R^n}f(x)e^{i\Phi(x)}dx$ is defined as:
$$\int^o _{\R^n}f(x)e^{i\Phi(x)}dx:=\lim_{\epsilon \to 0}\int_{\R^n} \phi(\epsilon x)f(x)e^{i\Phi(x)}dx$$
\end{definition}
In the case where $f\in L^1(\R^n)$ the oscillatory integral  reduces to a Lebesgue integral, i.e. $\int^o _{\R^n}f(x)e^{i\Phi(x)}dx =\int _{\R^n}f(x)e^{i\Phi(x)}dx$. \\
Definition \ref{def-int-osc} gives sense to classical Fresnel integrals such as
$\int_{\R^n}f(x)e^{\frac{i}{2}\|x\|^2} dx$ which are extensively applied in the theory of wave diffraction. In particular,  for $f=1$ definition \eqref{int-osc} yields the equality $\int_{\R^n}e^{\frac{i}{2}\|x\|^2} dx=(2\pi i)^{n/2}$.\\
In \cite{AlHKMa} oscillatory integration is generalized to the case where $\R^n$ is replaced by a real separable Hilbert space $(\Hi, \langle \;,\;\rangle)$ and the definition of {\em infinite dimensional Fresnel integral} is introduced. The construction relies upon a generalization of the Parseval equality 
\begin{equation}\label{Parseval1}
\int_{\R^n}\frac{e^{\frac{i}{2}\|x\|^2}}{(2\pi i)^{n/2}}f(x)dx=\int_{\R^n}e^{-\frac{i}{2}\|x\|^2}\hat f(x)dx,
\end{equation}
(valid for Schwartz test functions functions $f\in S(\R^n)$,  where $\hat f(x)=\int_{\R^n}e^{ixy}f(y)dy$). In fact (see \cite{ELT}) equality \eqref{Parseval1} can be generalized to the case the function $f:\R^n\to \C$ is the Fourier transform of a complex bounded Borel measure $\mu_f$ on $\R^n$, giving the following Parseval equality for the oscillatory integral 
\begin{equation}\label{Parseval2}
\int^o_{\R^n}\frac{e^{\frac{i}{2}\|x\|^2}}{(2\pi i)^{n/2}}f(x)dx=\int_{\R^n}e^{-\frac{i}{2}\|x\|^2}d\mu_f(x), 
\end{equation}
with $f(x)=\int_{\R^n}e^{ixy}d\mu_f(y)$. 
Formula \eqref{Parseval2} is crucial for the extension of oscillatory integration theory to an infinite dimensional setting.\\
Let us introduce the Banach space  $\Mi(\Hi)$ of complex Borel measures on $\Hi$ with finite total variation, endowed with the total variation norm $\|\mu\|_{\Mi(\Hi)}$. $\Mi(\Hi)$ is a commutative Banach algebra under convolution, the  unit being the Dirac point measure at $0$.\\ Let $\Fo(\Hi)$ be the space of complex functions $f:\Hi\to \C$  the form:
\begin{equation}\label{Fo}f(x)=\int_\Hi e^{i\langle x,y\rangle}d\mu(y)\equiv \hat \mu (x),\qquad x\in \Hi\end{equation}
for some $\mu \in \Mi(\Hi)$. The map $\Fo:\Mi(\Hi)\to \Fo(\Hi)$  sending a complex measure $\mu \in\Mi(\Hi)$ to its Fourier transform $\hat \mu$ defined by  Eq. \eqref{Fo} is linear and one to one. 
 By endowing the space $\Fo(\Hi)$  with the norm $\|f\|_\Fo:=\|\Fo^{-1}(f)\|_{\Mi(\Hi)}$, $\Fo(\Hi)$ becomes a commutative Banach algebra of continuous functions and the map $\Fo:\Mi(\Hi)\to \Fo(\Hi)$ is an isometry.\\
In \cite{AlHKMa,AlHK,ELT} the Parseval equality \eqref{Parseval2} is generalized to the case where $f\in \Fo(\Hi)$. The {\em infinite dimensional Fresnel integral} of a function $f\in \Fo(\Hi)$ is denoted by $\widetilde{ \int} e^{\frac{i}{2}\|x\|^2}f(x)dx$ and defined as
\begin{equation}\label{Parseval2}
\widetilde{\int} e^{\frac{i}{2}\|x\|^2}f(x)dx:=\int_{\Hi}e^{-\frac{i}{2}\|x\|^2}d\mu(x), 
\end{equation}
where $f(x)=\int_\Hi e^{i\langle x,y\rangle}d\mu(y)$ and
the right hand side of \eqref{Parseval2} is a well defined (absolutely convergent) Lebesgue integral.\\
%In \cite{ELT} the infinite dimensional Fresnel integral is obtained as limit of a sequence of finite dimensional approximations. Indeed it is possible to see \cite{ELT,Ma} that, given $f\in \Fo(\Hi)$, for any sequence $\{P_n\}_n$  of finite dimensional projection operators $P_n:\Hi\to\Hi$ such that $P_n(\Hi)\subset P_{n+1}(\Hi)$ $\forall n$ and $\lim_{n\to\infty }P_nx=x$ $\forall x\in \Hi$,  one has:
%\begin{eqnarray}\label{findim1}
% \tilde\int e^{\frac{i}{2}\|x\|^2}f(x)dx &=&\lim_{n\to\infty}\frac{1}{(2\pi i)^{n/2}}\int_{P_n\Hi}e^{\frac{i}{2}\|P_nx\|^2}f(P_nx)dP_nx\\
%&=&\int_{\Hi}e^{-\frac{i}{2}\|x\|^2}d\mu(x) .
%\end{eqnarray}
Infinite dimensional Fresnel integrals have been successfully applied to the representation of the solution of Schr\"odinger equation \eqref{schroedinger} (see i.e. \cite{AlHKMa,Ma} and references therein). Let us denote with $\Hi_t$  the real Hilbert space  of absolutely continuous paths $\gamma:[0,t]\to\R^d$, such that $\int_0^t\dot\gamma(s)^2ds<\infty$ and $\gamma(t)=0$. The inner  product  in $\Hi_t$ is defined as $\langle\gamma, \eta\rangle=\int_0^t\dot \gamma(s)\dot \eta(s)ds$. By assuming that the initial datum $u_0$ and the potential $V$ in Eq.  \eqref{schroedinger} belong to $\Fo(\R^d)$,  it is possible to prove that the function on $\Hi_t$:
$$\gamma\mapsto u_0(\gamma(0)+x)e^{-i\int_0^tV(\gamma(s)+x)ds},\qquad \gamma\in H_t,\; x\in \R^d,$$
belongs to $\Fo(\Hi_t)$.  Further  the infinite dimensional Fresnel integral
$$\widetilde{\int} e^{\frac{i}{2}\|\gamma\|^2}e^{-i\int_0^tV(\gamma(s)+x)ds}u_0(\gamma(0)+x)d\gamma$$ 
provides a functional integral representation of the solution to the Schr\"odinger equation \eqref{schroedinger}.\\
A partial generalization of the definition of infinite dimensional  Fresnel integrals and of formula \eqref{Parseval2}  was developed in \cite{AlMa2}, where  the quadratic phase function $\Phi(x)=\frac{i}{2}\|x\|^2$ was replaced with a fourth order polynomial. This new functional integral  allows the mathematical definition  of the Feynman path integrals for  the Schr\"odinger equation with a quartic-oscillator potential \cite{AlMa2,AlMa3,Ma2008}.\par
 In the following we are going to generalize the definition in \eqref{Parseval2} to polynomial phase functions  of any order and apply these {\em generalized Fresnel integrals} to the  construction of a Feynman-Kac formula for the solution of high-order heat-type equations \eqref{PDE-N}.

Let us consider   a real separable Banach space $(\Ba, \|\,\|)$.  Let $\Mi(\Ba)$ be the  space of complex bounded variation measures on $\Ba$, endowed with the total variation norm. As remarked above, $\Mi(\Ba)$ is a Banach algebra under convolution. Let $\Ba^*$ be the topological dual of $\Ba$ and  $\Fo(\Ba)$ the Banach algebra of complex-valued functions $f:\Ba^*\to \C$ of the form
\begin{equation}\label{Fo-Ba}f(x)=\int_{\Ba}e^{i\langle x,y\rangle}d\mu(y)\equiv \hat\mu(x), \qquad x\in \Ba^*,\, \mu\in \Mi(\Ba), \end{equation}
where $\langle \, ,\,\rangle$ denotes the dual pairing between $\Ba$ and $\Ba^*$. The space $\Fo(\Ba)$ endowed with the  norm $\|\hat \mu\|_{\Fo}:=\|\mu\|_{\Mi(\Ba)}$  and the pointwise multiplication is a Banach algebra of functions.\\
In the following we are going to define a class of linear continuous functionals on $\Fo(\Ba)$, by generalizing the construction of infinite dimensional Fresnel integrals defined by Eq. \eqref{Parseval2}.

\begin{definition}
Let $\Phi:\Ba\to \C$ be a continuous map  such that $\Rea (\Phi_p(x))\leq 0$ for all $x\in\Ba$.
 The infinite dimensional Fresnel integral on $\Ba^*$ with phase function $\Phi$ is the functional $I_{\Phi}:\Fo(\Ba)\to \C$, given by
\begin{equation}\label{Parseval-p}I_{\Phi}(f):=\int_\Ba e^{\Phi(x)}d\mu(x), \qquad f\in\Fo(\Ba), f=\hat\mu.\end{equation}
\end{definition}
By construction, the functional $I_{\Phi}$ is linear and continuous, indeed: 
$$|I_{\Phi_p}(f)|\leq \int_\Ba |e^{\Phi_p}|d|\mu|(x)\leq \|\mu\|=\|f\|_{\Fo }$$
Further $I_{\Phi}$ is normalized, i.e.,  $I_{\Phi_p}(1)=1$.
We summarize these properties  in the following proposition.
\begin{proposition}
The space $\Fo(\Ba)$ of Fresnel integrable functions is a Banach function algebra in the norm $\|\,\|_\Fo$. The infinite dimensional Fresnel integral with phase function $\Phi$ is a continuous bounded linear functional $I_{\Phi}:\Fo(\Ba)\to \C$ such that $|I_{\Phi_p}(f)|\leq \|f\|_{\Fo }$ and  $I_{\Phi_p}(1)=1$.
\end{proposition}

We can now present an interesting example of infinite dimensional Fresnel integral with polynomial phase function.\par
%For future references, we present here a particular example of the general previous general construction.
Fixed a $p\in \N$, with $p\geq 2$, let us consider the Banach space
$\Ba_p$  of absolutely continuous maps $\gamma:[0,t]\to\R$, with $\gamma(t)=0$ and  a weak derivative $\dot \gamma $ belonging to $L^p([0,t])$, endowed with the norm:
$$\|\gamma\|_{\Ba_p}=\left(\int_0^t|\dot \gamma(s)|^pds\right)^{1/p}. $$
The application  $T:\Ba_p\to L^p([0,t])$ mapping an element $\gamma\in \Ba_p$ to its weak derivative $\dot \gamma \in L^p([0,t])$ is an isomorphism and its inverse $T^{-1}:L^p([0,t])\to\Ba_p$  is given by:
\begin{equation}\label{T-1}
T^{-1}(v)(s)=-\int_s^t v(u)du\qquad v\in L^p([0,t]) .
\end{equation}
Analogously the dual space $\Ba_p^*$ is isomorphic to $  L^q([0,t])=(L^{p}([0,t]))^*$, with $\frac{1}{p}+\frac{1}{q}=1$, and the pairing $\langle \eta, \gamma\rangle$ between $\eta\in\Ba_p^*$ and $\gamma\in \Ba_p$ can be written in the following form:
$$\langle \eta, \gamma\rangle=\int_0^t\dot \eta(s)\dot\gamma(s)ds\qquad \dot\eta\in L_{q}([0,t]), \gamma\in\Ba_p.$$
Further $\Ba_p^*$ is isomorphic to $\Ba_q$.\\
Let us consider  the space $\Fo(\Ba_q) $ of functions $f:\Ba_q\to\C$ of the form 
$$f(\eta)=\int_{\Ba_p}e^{i\int_0^t\dot \eta(s)\dot \gamma(s)ds}d\mu_f(\gamma),\, \quad \eta \in \Ba_q, \mu_f\in \Mi(\Ba_p).$$
Let $\Phi_p:\Ba_p\to\C$ be the phase function defined as
$$\Phi_p(\gamma):=(-1)^p\alpha\int_0^t\dot\gamma(s)^pds,$$
where  $\alpha\in\C$  is a complex constant such that 
\begin{itemize}
\item $\Rea(\alpha)\leq 0$ if $p$ is even,
\item $\Rea(\alpha)= 0$ if $p$ is odd.
\end{itemize}
%In particular, in the case where $p$ is even the phase function $\Phi_p$ is proportional to the $\Ba_p$-norm. \\
The  infinite dimensional Fresnel integral on $B_q$ with phase function $\Phi_p$ is the functional $I _{\Phi_p} : \Fo(\Ba_q)\to\C$  given by
\begin{equation}\label{funzI-p}
I _{\Phi_p}(f)=\int_{\Ba_p}e^{(-1)^p\alpha \int_0^t\dot \gamma(s)^pds}d\mu_f(\gamma), \quad f\in \Fo(\Ba_q),\,f=\hat\mu_f.
\end{equation}
%In the following section we shall apply this functional to the construction of the solution of heat-type equations of order $p$.

The following lemma states an interesting connection between the functional \eqref{funzI-p} and the high-order PDE \eqref{PDE-p}.
\begin{lemma}\label{lemmacyl}
Let  $f: \Ba_q\to\C$ be a cylinder function of the 
 following form:
$$f(\eta)=F(\eta(t_1), \eta(t_2), ...,\eta(t_n)),\qquad \eta \in \Ba_q,$$ with $0\leq t_1<t_2<...<t_n< t$ and $F:\R^n\to \C$, $F\in \Fo(\R^n)$: 
$$F(x_1,x_2, ..., x_n)=\int_{\R^n}e^{i\sum_{k=1}^ny_kx_k}d\nu_F(y_1,...,y_n), \qquad \nu_F\in\Mi(\R^n).$$
Then $f\in\Fo(\Ba_p)$ and its infinite dimensional Fresnel integral with phase function $\Phi_p$ is given by
\begin{equation}\label{IntIphi-p}
I _{\Phi_p}(f)=\int^o_{\R^n}F(x_1,x_2, ...,x_n)\Pi_{k=1}^nG^p_{t_{k+1}-t_{k}}(x_{k+1},x_{k})dx_1...dx_n,
\end{equation}
where $x_{n+1}\equiv 0$, $t_{n+1}\equiv t$ ,  $G^p_s$ is the fundamental solution \eqref{Green} of the high order heat-type equation \eqref{PDE-p} and the integral on the right hand side of \eqref{IntIphi-p} is an oscillatory integral in the sense of definition \ref{def-int-osc}. 
\end{lemma}
\begin{Remark}
In the case $p$ is even and $\Rea(\alpha) <0$ the integral \eqref{IntIphi-p}  is an absolutely convergent Lebesgue integral because of the boundedness of the function $F\in \Fo(\R^n)$ and the summability of the function $g^p_t$ stated in lemma \ref{lemma1-asy}. %stating the $L^1$-summability of the function $g^p_t$ defined by \eqref{Green2},  assures that the integral in \eqref{IntIphi-p} is absolutely convergent. Indeed, we have:
%\begin{multline}
%\int_{\R^n}|F(x_1,x_2, ...,x_n)|\, \Pi_{k=1}^n|G^p_{t_{k+1}-t_{k}}(x_{k+1},x_{k})|dx_k\\
%\leq \|F\|_{\Fo(\R^n)}\int_{\R^n}\Pi_{k=1}^n|g^p_{t_{k+1}-t_{k}}(y_k)|dy_k<\infty\end{multline}
%where $y_k=x_{k+1}-x_{k}$ and we have used the inequality $\|F\|_\infty\leq \|F\|_{\Fo(\R^n)}$.
\end{Remark}
\begin{proof}[Proof of lemma \ref{lemmacyl}]
The proof that $f\in\Fo(\Ba_p)$ follows froms the explicit form of the function $f$
$$f(\eta)=F(\eta(t_1), \eta(t_2), ...,\eta(t_n))=\int_{\R^n}e^{i\sum_{k=1}^ny_k\eta(t_k)}d\nu_F(y_1,...,y_n)),\quad \eta\in \Ba_q.$$
and the identity
$$e^{iy\eta(s)}=\int _{\Ba_p}e^{i\langle \eta,\gamma \rangle}\delta_{yv_{s}}(\gamma),$$ where $v_s\in\Ba_p$ is the vector of $\Ba_p$ defined by $$\langle \eta, v_s\rangle=\eta(s), \qquad \forall \eta \in \Ba_q,$$ which can be explicitly written as $$v_s(\tau)=\chi_{[0,s]}(t-s)+\chi_{(s,t]}(t-\tau)s.$$
By the definition of the functional $I_{\Phi_p}$ we have
\begin{eqnarray}
I _{\Phi_p}(f)
&=&\int_{\R^n} e^{(-1)^p\alpha \int_0^t \left(\sum_{k=1}^ny_k\dot v_{t_k}(\tau)\right)^p d\tau}d\nu_F(y_1,...,y_n)\nonumber\\
&=&\int_{\R^n} e^{\alpha \int_0^t \left(\sum_{k=1}^ny_k\chi_{(t_k,t]}(\tau)\right)^p d\tau}d\nu_F(y_1,...,y_n)\nonumber\\
&=&\int_{\R^n} e^{\alpha \int_0^t \left(\sum_{k=1}^n\chi_{(t_k,t_{k+1}]}(\tau)\sum_{j=1}^ky_j\right)^p d\tau}d\nu_F(y_1,...,y_n)\nonumber\\
&=&\int_{\R^n} e^{\alpha\sum_{k=1}^n    ( \sum_{j=1}^k y_j)^p(t_{k+1}-t_{k})}d\nu_F(y_1,...,y_n)\label{eq16}
\end{eqnarray}
On the other hand the last line of Eq. \eqref{eq16} coincides with the oscillatory integral
\begin{equation}\int^o_{\R^n}F(x_1,x_2, ...,x_n)\Pi_{k=1}^nG^p_{t_{k+1}-t_{k}}(x_{k+1},x_{k})dx_1...dx_n.
\end{equation}
Indeed, taken an arbitrary test function $\phi\in S(\R^n)$ such that $\phi(0)=1$, the the function $F_\epsilon:\R^n\to\C$
$$F_\epsilon (x_1,x_2, ...,x_n)\equiv F(x_1,x_2, ...,x_n)\phi(\epsilon x_1,\epsilon x_2, ...,\epsilon x_n))\Pi_{k=1}^nG^p_{t_{k+1}-t_{k}}(x_{k+1},x_{k})$$
is summable because of the boundedness of $F\in\Fo(\R^n)$ and the decaying properties at infinity stated in lemma \ref{lemma1-asy}. Further a change of variable argument and Fubini theorem   yield:
$$\int_{\R^n}F_\epsilon(x)dx 
=\int_{\R^n}  \left(  \int_{\R^n}e^{\alpha\sum_{k=1}^n(t_{k+1}-t_k)(\sum_{j=1}^ky_j+\epsilon \xi_j)^p}\hat \phi (\xi
)d\xi\right) d\nu_F(y),$$
where $\phi(x)=\int_{\R^n}e^{ix \xi }\hat \phi (\xi)d\xi$. By dominated convergence theorem and the condition $\phi(0)=\int_{\R^n}\hat \phi (\xi)d\xi=1$, we eventually obtain
$$\lim_{\epsilon\to 0}\int_{\R^n}F_\epsilon(x)dx =\int_{\R^n}  e^{\alpha\sum_{k=1}^n(t_{k+1}-t_k)(\sum_{j=1}^ky_j)^p}d\nu_F(y),$$
%&=&\int_{\R^n}\left( \int_{\R^n}e^{-i\sum_{k=1}^n\xi_k\sum_{l=1}^ky_l}\Pi_{k=1}^ng^p_{t_{k+1}-t_{k}}(\xi_k)d\xi_k\right)d\nu_F(y_1,...,y_n)\nonumber\\
%&=& \int_{\R^n}e^{\alpha \sum_{k=1}^n (\sum_{l=1}^ky_l)^p(t_{k+1}-t_{k})} d\nu_F(y_1,...,y_n).\nonumber
%\end{eqnarray}

\end{proof}

\begin{corollary}\label{cor1}
Let $u_0\in \Fo(\R)$. Then the cylinder function $f_0:\Ba_q\to\C$ defined by
$$f_0(\eta):=u_0(x+\eta(0)), \qquad x\in \R, \eta\in \Ba_q,$$
belongs to $\Fo(\Ba_q)$ and its  infinite dimensional Fresnel integral with phase function $\Phi_p$ provides a representation for the solution of the Cauchy problem \eqref{PDE-p}, in the sense that the function 
$u(t,x):=I_{\Phi_p}(f_0)$ has the form \begin{equation}\label{sol-free}u(t,x)=\int^o _\R G_t(x,y)u_0(y)dy.\end{equation}
In the case $p$ is even and $\Rea(\alpha)<0$ then the integral \eqref{sol-free} is absolutely convergent, while in the general case it is meant in the oscillatory sense of definition \ref{def-int-osc}.
\end{corollary}

\section{A generalized Feynman-Kac formula}
In the present section,  we consider a Cauchy problem of the form 
\begin{equation} \label{PDE-p-V}\left\{ \begin{aligned}
\frac{\partial}{\partial t}u(t,x)&=(-i)^p \alpha \frac{\partial^p}{\partial x^p}u(t,x)+V(x)u(t,x)\\
u(0,x) &= u_0(x),\qquad x\in\R, t\in [0,+\infty)
\end{aligned}\right. \end{equation}
where $p\in\N$, $p\geq 2$, and $\alpha\in\C$ is a complex constant such that $|e^{\alpha tx^p}|\leq 1$ forall $x\in\R, t\in [0,+\infty)$, while $V:\R\to\C$ is a bounded continuous function. Under these assumption the Cauchy problem \eqref{PDE-p-V} is well posed in $L^2(\R)$. Indeed the operator $\di_p:D(\di_p)\subset L^2(\R)\to L^2(\R)$ defined by 
\begin{eqnarray*}
D(\di_p)&:=& H^p=\{u\in L^2(\R), k\mapsto k^p\hat u(k)\in L^2 (\R) \},\\
\widehat{\di_pu}(k)&:=&k^p\hat u(k), \, u\in D(\di_p),
\end{eqnarray*}
($\hat u$ denoting the Fourier transform of $u$) is self-adjoint. For $\alpha\in\C$, with $|e^{\alpha tx^p}|\leq 1$ forall $x\in\R, t\in [0,+\infty)$, one has that  the operator $A:=\alpha D_p$ generates a strongly continuous semigroup $(e^{tA})_{t\geq 0}$ on $L^2(R)$. By denoting with $B:L^2(\R)\to L^2(\R)$ the bounded multiplication operator defined by 
$$Bu(x)=V(x)u(x), \qquad u\in L^2(\R),$$
one  has that the operator sum $A+B:D(A)\subset L^2(\R)\to L^2(\R)$ generates a   strongly continuous semigroup $(T(t))_{t\geq 0}$ on $L^2(\R)$. Moreover, given a $u\in L^2(\R)$, the vector $T(t)u$ 
can be computed by means of the convergent (in the $L^2(\R)$-norm) Dyson series (see \cite{HiPhi}, Th. 13.4.1):
\begin{equation}\label{Dyson}
T(t)u=\sum_{n=0}^\infty S_n(t)u,
\end{equation}
where $S_0(t)u=e^{tA}u$ and $S_n(t)u=\int_0^t e^{(t-s)A}VS_{n-1}(s)uds$.
By passing to a subsequence, the series above converges also a.e. in $x\in\R$ giving
\begin{multline}\label{Dyson2}
T(t)u (x)=\\
=\sum_{n=0}^\infty \;\; \idotsint\limits _{0\leq s_1 \leq \dots \leq s_n \leq t }\int_{\R^{n+1}}V(x_1)\dots V(x_n) G_{t-s_n}(x,x_n)
G_{s_n-s_{n-1}}(x_n,x_{n-1})\\ \dots G_{s_1}(x_1,x_0)
u_0(x_0)dx_0\dots dx_n\,ds_1\dots ds_n ,\qquad a.e.\; x\in \R.
\end{multline}
Under suitable assumptions on  the initial datum $u_0$ and the potential $V$, we are going to construct a representation of the solution of equation \eqref{PDE-p-V} in $L^2(\R)$ in terms of an infinite dimensional oscillatory integral with polynomial phase.
\begin{teorema}
Let $u_0\in \Fo(\R)\cap L^2(\R)$ and $V\in \Fo(\R)$, with $u_0(x)=\int_\R e^{ixy}d\mu_0(y)$ and $V(x)=\int_\R e^{ixy}d\nu(y)$, $\mu_0,\nu\in\Mi(\R)$.  Then the functional $f_{t,x}:\Ba_q\to\C$ defined by
\begin{equation}\label{f-tx}
f_{t,x}(\eta):=u_0(x+\eta(0))e^{\int_0^tV(x+\eta(s))ds}, \qquad x\in \R, \eta\in \Ba_q,
\end{equation}
belongs to $\Fo(\Ba_q)$ and its  infinite dimensional Fresnel integral with phase function $\Phi_p$ provides a representation for the solution of the Cauchy problem \eqref{PDE-p-V}.
\end{teorema}
\begin{Remark}
By Plancherel's theorem the assumption that $u_0\in \Fo(\R)\cap L^2(\R)$ is equivalent to the fact that $u_0$ is the Fourier transform of a function $\hat u_0\in L^1(\R)\cap L^2(\R)$.
\end{Remark}
\begin{proof}

Let $\mu_V\in \Mi(\Ba_p)$ be the measure  defined by 
$$\int_{\Ba_p}f(\gamma)d\mu_V(\gamma)=\int_0^t\int_\R e^{ixy}f(y\,v_s)d\nu(y)ds, \qquad f\in C_b(\Ba_p),$$
where $v_s\in \Ba_p$ is the function  $v_s(\tau)=\chi_{[0,s]}(\tau)(t-s)+\chi_{(s,t]}(t-\tau)s$. One can easily verify   that $\|\mu_V\|_{\Mi(\Ba_p)}\leq t\|\nu\|_{\Mi(\R)}$ and the map $\eta \in \Ba_q\mapsto \int_0^tV(x+\eta(s))ds$ is the Fourier transform of $\mu_V$. Analogously the map $\eta \in \Ba_q\mapsto \exp(\int_0^tV(x+\eta(s))ds) $ is the Fourier transform of the measure $\nu_V\in \Mi(\Ba_p)$ given by $\nu_V=\sum _{n=0}^\infty\frac{1}{n!}\mu_V^{*n}$, where  $\mu_V^{*n}$ denotes the $n$-fold convolution of $\mu_V$ with itself.    %  namely  for any $f\in C_b(\Ba_p)$:
%$$\int_{\Ba_p}f(\gamma)d\nu_V(\gamma)=\sum_{n=0}^\infty\frac{1}{n!} \int_{B_p\times B_p\times...\times B_p}f(\gamma_1+...+\gamma_n)d\mu_V(\gamma_1)\cdots d\mu_V(\gamma_n).$$
The series is convergent in the $\Mi(\Ba_p)$-norm and one has $\|\nu_V\|_{\Mi(\Ba_p)}\leq e^{t\|\nu\|_{\Mi(\R)}}$. Further,  by lemma \ref{lemmacyl} the cylinder function $\eta\mapsto u_0(x+\eta(0))$, $\eta \in \Ba_q$, is an element of $\Fo(\Ba_q)$. More precisely, it is the Fourier transform of the measure $\nu_{u_0}$ defined by
$$\int_{\Ba_p}f(\gamma)d\nu_{u_0}(\gamma)=\int _\R e^{ixy}f(y\,v_0)d\mu_0(y), \qquad f\in  C_b(\Ba_p).$$
We can then conclude that the map 
 $f_{t,x}:\Ba_q\to\C$ defined by \eqref{f-tx} belongs to $\Fo(\Ba_q)$ and its infinite dimensional Fresnel integral $  I_{\Phi_p}(f_{t,x})$ with phase function $\Phi_p$ is given by
\begin{multline}\nonumber
\sum_{n=0}^\infty \frac{1}{n!}\int_{\Ba_p}e^{(-1)^p\alpha \int_0^y \dot \gamma(s)^pds}d\nu_{u_0}*\mu_V*\cdots *\mu_V=\\
 \sum_{n=0}^\infty \frac{1}{n!}\int_0^t...\int_0^t I_{\Phi_p}\big(u_0(x+\eta(0))V(x+\eta(s_1))\dots V(x+\eta(s_n))\big)ds_1 \cdots ds_n
\end{multline}
By the symmetry of the integrand the latter is equal to
$$\sum_{n=0}^\infty\; \idotsint\limits _{0\leq s_1 \leq \dots \leq s_n \leq t } I_{\Phi_p}\big(u_0(x+\eta(0))V(x+\eta(s_1))\dots V(x+\eta(s_n))\big)ds_1 \cdots ds_n$$
By lemma \ref{lemmacyl} we eventually obtain
\begin{multline}\nonumber \sum_{n=0}^\infty\; \idotsint\limits _{0\leq s_1 \leq \dots \leq s_n \leq t }\int_{\R^{n+1}} u_0(x+x_0)V(x+x_1)\dots V(x+x_n)G_{s_1}(x_1,x_0)\\G_{s_2-s_1}(x_2,x_1)\dots G_{t-s_n}(0,x_{n})dx_0 dx_1\cdots dx_n ds_1 \cdots ds_n,
\end{multline}
that coincides with the  Dyson series  \eqref{Dyson2}  for the solution of the high-order PDE \eqref{PDE-p-V}, as one can easily verify by means of a change of variables argument.
\end{proof}
\section*{Acknowledgments}
Many interesting discussions with Prof. S. Albeverio, S. Bonaccorsi, G. Da Prato  and L.Tubaro are gratefully acknowledged, as well as the financial support of CIRM-Fondazione Bruno Kessler to the project { \em  Functional integration and applications to quantum dynamical systems}.

\end{document}